\documentclass[11pt]{article}

\PassOptionsToPackage{dvipsnames}{xcolor}

\usepackage{caption}
\usepackage{graphicx,comment}
\usepackage{graphicx}
\usepackage{multicol}
\usepackage{hyperref}
\usepackage{stmaryrd}
\usepackage{amssymb,amsmath,bm}
\usepackage{amsmath}
\usepackage{mathrsfs}
\usepackage{slashed}
\usepackage{amssymb}
\usepackage{amsthm}
\usepackage{graphicx,overpic}
\usepackage{svg}
\usepackage{epstopdf}
\usepackage{enumerate}
\usepackage{comment}
\usepackage{indentfirst}
\usepackage{setspace}  

\usepackage{textcomp}
\usepackage{enumerate}      
\usepackage{graphicx} 
\usepackage{caption}
\usepackage{subcaption}

\usepackage{tabu}

\usepackage[export]{adjustbox}

\usepackage{mathrsfs}

\usepackage{url}

\usepackage{hyperref}

\def\psl{\mathrm{PSL}(2,\mathbb{R})}
\def\pgl{\mathrm{PGL}(2,\mathbb{R})}

\def\SL{\mathrm{SL}(2,\mathbb{R})}

\def\sl{\mathfrak{sl}_2(\mathbb{R})}
\def\fund{\pi_1(\Sigma)}
\def\hom{\mathrm{Hom}\big(\fund, \psl\big)}

\def\mcg{\mathrm{Mod}(\Sigma)}

\def\univcover{\widetilde{\SL}}
\def\Hyp{\mathrm{Hyp}}
\def\Par{\mathrm{Par}}
\def\Parp{\mathrm{Par}^+}
\def\Parm{\mathrm{Par}^-}
\def\Ell{\mathrm{Ell}}

\def\R{\mathcal{R}(\Sigma)}

\def\Rns{\mathcal{R}_n^s(\Sigma)}

\def\HP{\mathrm{HP}}

\def\HPns{\mathrm{HP}_n^s}
\def\NPns{\mathrm{NP}^s_n(\Sigma)}

\def\ev{\mathrm{ev}}
\def\tr{\mathrm{tr}}

\def\parpp
{\begin{bmatrix}
    1 & 1 \\
    0 & 1
\end{bmatrix}}

\def\parmp
{\begin{bmatrix}
    1 & -1 \\
    0 & 1
\end{bmatrix}}

\def\part
{\begin{bmatrix}
    1 & t \\
    0 & 1
\end{bmatrix}}

\def\ell
{\begin{bmatrix}
    \cos\frac{\theta}{2} & \sin\frac{\theta}{2} \\
    -\sin\frac{\theta}{2} & \cos\frac{\theta}{2}
\end{bmatrix}}

\def\elltheta
{\begin{bmatrix}
    \cos\theta & \sin\theta \\
    -\sin\theta & \cos\theta
\end{bmatrix}}

\def\path{\{\phi_t\}\interval}
\def\interval{_{t\in [0,1]}}

\newtheorem{theorem}{Theorem}[section]
\newtheorem{proposition}[theorem]{Proposition}
\newtheorem{lemma}[theorem]{Lemma}
\newtheorem{corollary}[theorem]{Corollary}
\newtheorem{conjecture}[theorem]{Conjecture}
\newtheorem{remark}[theorem]{Remark}

\theoremstyle{definition}
\newtheorem{definition}[theorem]{Definition}

\newtheorem{question}[theorem]{Question}

\begin{document}
	
\title{\textbf{On totally hyperbolic non-Fuchsian type-preserving representations}}
\author{Inyoung Ryu}
\date{}
\maketitle

\begin{abstract}
We identify type-preserving representations $\phi: \pi_1(\Sigma)\to \psl$ of the fundamental group of \textcolor{black}{every} punctured surface $\Sigma = \Sigma_{g,p}$ that are not Fuchsian yet send all non-peripheral simple closed curves to hyperbolic elements, which give \textcolor{black}{a negative answer} to a question of Bowditch. These representations have relative Euler class $e(\phi) = \pm (\chi(\Sigma) + 1)$, and their $\psl$-conjugacy classes form a full-measure subset of $2p$ connected components of the relative character variety.
We further show that, while these representations are not Fuchsian, their restrictions to certain subsurfaces of $\Sigma$ are Fuchsian.

\end{abstract}



\section{Introduction}\label{intro}

    In this paper, we consider punctured surfaces, which are surfaces obtained from connected, closed, oriented surfaces by removing finitely many points. 
    Let $\Sigma = \Sigma_{g,p}$ be a punctured surface of genus $g$ and $p\geqslant 1$ punctures, with Euler characteristic $\chi(\Sigma) < 0$.
    We study representations from the fundamental group $\fund$ of the surface $\Sigma$ to $\psl$. 
    In particular, we consider \emph{type-preserving} representations, \textcolor{black}{i.e.}, representations $\phi: \pi_1(\Sigma)\to \psl$ sending all the peripheral elements (elements represented by loops around a puncture) of $\pi_1(\Sigma)$ to parabolic elements of $\psl.$ We focus on type-preserving representations that are non-elementary, meaning those whose image is Zariski dense in $\psl$. 
    The relative Euler class $e(\phi)$ of a type-preserving representation $\phi$ is the obstruction class of the associated flat $\psl$-bundle over $\Sigma,$ which is an integer invariant under the isomorphism $\mathrm{H}^2(\Sigma, \partial \Sigma;\pi_1(\psl))\cong \mathbb Z.$
     The relative Euler class satisfies \textcolor{black}{the} Milnor-Wood inequality
     $$\chi(\Sigma)\leqslant e(\phi)\leqslant -\chi(\Sigma),$$
     where the equality $|e(\phi)| = -\chi(\Sigma)$ holds if and only if $\phi$ is \emph{Fuchsian}, that is, discrete and faithful. \textcolor{black}{(See \cite{yang}, Appendix A.)}
     \medskip 

    The main focus of this paper is on representations satisfying the following property:
    \begin{definition}
        A representation $\phi: \pi_1(\Sigma)\to \psl$ is \emph{totally hyperbolic} if it 
        sends every non-peripheral element of $\fund$ represented by a simple closed curve to a hyperbolic element of $\psl$.
    \end{definition}
     
    It is well known that every Fuchsian type-preserving representation is totally hyperbolic. As a converse, Bowditch asked the following question.

    \begin{question}\cite[Question C]{bowditch}
        If a non-elementary type-preserving representation $\phi:\pi_1(\Sigma)\to \psl$ is totally hyperbolic, must $\phi$ be Fuchsian?
    \end{question}
    This question is closely related to a conjecture of Goldman regarding the mapping class group action on the $\psl$-character variety. 
    For a closed surface $\Sigma$, recall that the $\psl$-character variety is the \textcolor{black}{space} $\mathcal{M}(\Sigma) = \hom \textcolor{black}{\sslash} \psl$ where the quotient \textcolor{black}{is the GIT quotient.}
    The connected components of $\mathcal{M}(\Sigma)$ are indexed by the \textcolor{black}{Euler classes} \cite{goldman};
    and the mapping class group $\mcg$, the group of isotopy classes of orientation preserving homeomorphisms $f: \Sigma\to \Sigma$, acts on each connected component of $\mathcal{M}(\Sigma)$. 
    The extremal components characterized by the Euler classes $\pm \chi(\Sigma)$, which consist of Fuchsian representations, are
    respectively identified with the Teichm\"uller space of $\Sigma$. It is well known that the action of $\mcg$ on these \textcolor{black}{two} components is properly discontinuous. On each of the other non-extremal components, Goldman conjectured in \cite{goldman_conjecture} that 
    $\mcg$ acts ergodically with respect to the measure induced by Goldman symplectic form \cite{goldman_measure}. 
    Marché and Wolff \cite{MW} proved 
    that, when the surface $\Sigma$ is closed, an affirmative answer to Bowditch's question\textcolor{black}{---which can also be asked for closed surfaces---}implies that Goldman's conjecture is true. They also gave an affirmative answer to Bowditch's question for the closed oriented surface $\Sigma = \Sigma_{2,0}$ of genus two, thereby proving Goldman's conjecture in the genus two case \textcolor{black}{\cite{MW, MW2}}. It seems to be of folklore knowledge that, for a punctured surface $\Sigma = \Sigma_{g,p}$, Bowditch's question and Goldman's conjecture are similarly related via the relative character variety of type-preserving representations.

    For punctured surfaces with Euler characteristic $\chi(\Sigma) = -1$, namely, the \textcolor{black}{once-punctured} torus and the \textcolor{black}{thrice-punctured} sphere, Bowditch's question has an affirmative answer \textcolor{black}{for the trivial reason that} 
    every non-elementary type-preserving representation has relative Euler class $\pm 1$---and is therefore Fuchsian. 
    The first \textcolor{black}{set of} counterexamples \textcolor{black}{to Bowditch's question} were identified by Yang in \cite{yang}, in the case of the four-holed sphere $\Sigma = \Sigma_{0,4}$, with relative Euler class $\pm 1$. For other punctured surfaces, the question has remained open. 
    \smallskip
    
    The main result of this paper gives counterexamples to Bowditch's question for \textcolor{black}{every} punctured surfaces $\Sigma = \Sigma_{g,p}$ with $\chi(\Sigma)\leqslant -2$. 
    To state the result, we recall that there is an invariant $s(\phi)\in \{\pm 1\}^p$ of each type-preserving representation $\phi$\textcolor{black}{,} called the \emph{sign} of $\phi$, which assigns each puncture of $\Sigma$  a sign $+1$ or $-1.$ (See Section \ref{sign} for details.) 
    For the sign $s\in\{\pm 1\}^p,$ we let $p_+(s)$ be the number of $+1$'s and let $p_-(s)$ be the number of $-1$'s in the components of $s.$
    For $n\in \mathbb{Z}$ and $s\in \{\pm 1\}^p$, let
    $\mathcal{M}^s_n(\Sigma)$ be a subspace of the relative character variety $\mathcal{M}(\Sigma)$ which consists of the \textcolor{black}{$\psl$-conjugacy} classes of type-preserving representations with relative Euler class $n$ and sign $s$.
    The main result in \cite{RY} shows that, 
    if $n$ and $s$ \textcolor{black}{satisfy} the generalized Milnor-Wood inequality 
    $$\chi(\Sigma) + p_+(s)\leqslant n\leqslant -\chi(\Sigma) - p_-(s),$$ 
    then the space $\mathcal{M}^s_n(\Sigma)$ is non-empty and forms a connected component of $\mathcal{M}(\Sigma)$. The following Theorem \ref{thm_counter_Bowditch} shows that there exist totally hyperbolic, non-Fuchsian representations whose $\psl$-conjugacy classes form a dense subset of $2p$ connected components of $\mathcal{M}(\Sigma)$.
    
     \medskip
    \begin{theorem}\label{thm_counter_Bowditch}
        Let $\Sigma = \Sigma_{g,p}$ with $\chi(\Sigma)\leqslant -2$ and $p\geqslant 1$,  
        \textcolor{black}{and assume $n \in \mathbb{Z}$ and $s\in \{\pm 1\}^p$ satisfy one of the following:} \begin{enumerate}[(1)] 
        \item $n = -\chi(\Sigma) - 1$ and $p_-(s) = 1$, and
        
        \item $n = \chi(\Sigma) + 1$ and $ p_+(s) = 1$
        \end{enumerate}
        Then there exist uncountably many non-Fuchsian type-preserving representations $\phi:\pi_1(\Sigma)\to \psl$ of relative Euler class $e(\phi) = n$ and sign $s(\phi) = s$ that 
        are \textcolor{black}{totally hyperbolic.
        Furthermore, the set of their $\psl$-conjugacy classes has full measure in $\mathcal{M}^s_n(\Sigma)$.}
    \end{theorem}

    \begin{remark}
        \textcolor{black}{When $\Sigma = \Sigma_{0,4}$}, the totally hyperbolic representations in Theorem \ref{thm_counter_Bowditch} correspond to those identified in \cite{yang}.
    \end{remark}

    Although these totally hyperbolic representations are not Fuchsian, they exhibit certain similarities to Fuchsian representations, which will also be discussed in this paper. \textcolor{black}{We hope this observation will contribute to a better understanding of the mapping class group action on these components of the relative character variety.}
    \begin{corollary}\label{cor_almost_fuchsian}
        Let $\phi$ be a non-Fuchsian totally hyperbolic representation identified in Theorem \ref{thm_counter_Bowditch}. 
        Then there exists a pair of pants $P$ in $\Sigma$ such that,
        for each connected component $\Sigma'$ of the complement $\Sigma\setminus P$,
        the restriction $\phi|_{\pi_1(\Sigma')}$ is Fuchsian.
    \end{corollary}

    \noindent\textbf{Acknowledgments.} 
    The author would like to thank her supervisor Tian Yang for the guidance and support. The author would also like to thank Viraj Joshi for helpful discussions.
\section{Preliminary}\label{prelim}

    Let $\Sigma = \Sigma_{g,p}$ be a punctured surface with genus $g$ and $p$ punctures, and let $\fund$ be the fundamental group of $\Sigma$. 
    A \emph{primitive peripheral element} of $\fund$ 
    is an element represented by a closed curve freely homotopic to a circle around a puncture of $\Sigma$. Notice that each puncture of $\Sigma$ determines a conjugacy class of $\pi_1(\Sigma)$ consisting of the primitive peripheral elements around it.
    
    Let $c_1,\cdots, c_p$ be \textcolor{black}{the} primitive peripheral elements of $\pi_1(\Sigma)$, chosen for each puncture.
    We denote by $\mathrm{HP}(\Sigma)$ the subspace of $\hom$ consisting of representations $\phi:\pi_1(\Sigma)\to \psl$ such that, for each $i\in\{1,\dots,p\},$ $\phi(c_i)$ is either a hyperbolic or a parabolic element of $\psl$. 
    In particular, when $\phi(c_i)$ is parabolic for every $i\in\{1,\dots,p\},$
    we call $\phi$ \emph{type-preserving}; and we denote by $\R$ the subspace of $\hom$ consisting of type-preserving representations.
    \smallskip

    In this section, we recall the definitions and basic properties of the relative Euler classes and the signs of representations in $\HP(\Sigma)$ together with the results from \cite{RY}, which serve as the key ingredients in our main results. 
    In the rest of the paper, we will use the notation $\pm A$ for a $\psl$-element, $A$ for an $\SL$-element, and $\widetilde{A}$ for an element in the universal cover $\univcover$.
    
    \subsection{The relative Euler classes}\label{Euler}

    We first recall the definition of the relative Euler class. 
    For each type-preserving representation $\phi$, there is a corresponding flat principal $\psl$-bundle over $\Sigma$ with parabolic boundary condition, for which $\phi$ is the holonomy representation. 
     The relative Euler class of this bundle is an obstruction class in  $\mathrm{H}^2\big(\Sigma, \partial \Sigma ; \pi_1\big(\psl\big)\big)$, 
     which measures the obstruction to extend the trivialization on $\partial \Sigma$ to a global trivialization on $\Sigma$. 
     This obstruction class defines the relative Euler class $e(\phi)$ of the representation $\phi$, which can be identified with an integer via the isomorphism $\mathrm{H}^2\big(\Sigma, \partial \Sigma ; \pi_1\big(\psl\big)\big)\cong \mathbb{Z}.$
     \medskip

    For more details, we recall the structure of $\psl$ and its universal cover $\univcover$. 
    Recall that elements of $\psl$ are classified as hyperbolic, parabolic and elliptic as follows: For a non-trivial $\pm A$ in $\psl,$ let $A$ be one of its lifts in $\SL.$ Then $\pm A$ is hyperbolic if $|\tr(A)|>2,$ parabolic if $|\tr(A)|=2,$ and elliptic if $|\tr(A)|<2.$ 
    We respectively let $\Hyp,$ $\Par$ and $\Ell$ be the spaces of hyperbolic, parabolic and elliptic elements of $\psl$. 
    $\Par$ consists of two disjoint subsets $\Parp$ and $\Parm$, which are respectively the $\psl$-conjugacy classes of $\pm \parpp$ and $\pm \parmp$. We call a parabolic element \emph{positive} if it is in $\Par^+$, and \emph{negative} if it is in $\Par^-$.
  Similarly, non-central elements of the universal covering $\univcover$ of $\psl$ are also classified as \emph{hyperbolic}, \emph{parabolic}, or 
    \emph{elliptic} according to the type of their projections to $\psl$; and a parabolic element of $\univcover$ is \emph{positive} or \emph{negative} if its projection to $\psl$ is respectively so. As a convention, we define the \emph{trace} of an element in $\univcover$ to be the trace of its projection to $\SL.$ Then a non-central element in $\univcover$ is hyperbolic, parabolic, or elliptic if its trace is greater than, equal to, or less than $2$ in the absolute value.  
    \smallskip

    As $\psl$ is homeomorphic to an open solid torus, its universal covering 
    group $\univcover$ is  homeomorphic to $\mathbb{R}^3$, with the group of
    deck transformations 
    $\pi_1(\psl)\cong \mathbb{Z}$. 
    The preimage of each of $\Hyp$, $\Par^\pm$ and $\Ell$ in $\univcover$ consists of infinitely many connected components, indexed by the integers. We denote these components as follows: 
    \smallskip
    
    Recall that the center of  $\univcover$ consists of the lifts of $\pm\mathrm I$, and is isomorphic to $\pi_1\big(\psl\big)\cong \mathbb{Z}$.  Let $\widetilde{\exp}: \sl\to \univcover$ be the exponential map of $\univcover$, and let $z \doteq \widetilde{\exp}\begin{bmatrix}
        0 & \pi \\
       -\pi & 0 
    \end{bmatrix}$. Then $z$ is a generator of the center of $\univcover$, and every lift of $\pm\mathrm{I}$ in $\univcover$ is equal to $z^n = \widetilde{\exp}
    \begin{bmatrix}
        0 & n\pi \\
       -n\pi & 0 
    \end{bmatrix}$ for some $n\in\mathbb Z$. 
Let  $\Hyp_0$ be the space of hyperbolic elements of $\univcover$ which lie in one-parameter subgroups; 
    and for any $n\in\mathbb Z$, let  $\Hyp_n \doteq z^n\Hyp_0$. 
    Then $\{\Hyp_n\}_{n\in \mathbb Z}$ are the connected components of the space of hyperbolic elements in $\univcover$; and two hyperbolic elements in $\univcover$ are conjugate if and only if they lie in the same connected component $\Hyp_n$  for some $n\in\mathbb Z$ and have the same trace.
    Similarly, let $\Par^+_0$ and $\Par^-_0$ respectively be the spaces of positive and negative parabolic elements of $\univcover$ lying in the image  of the exponential map, and let  $\Par_0
    = \Par^+_0\cup \Par^-_0.$ For any $n\in\mathbb Z,$ let $\Par^+_n = z^n\Par^+_0$ and $\Par^-_n = z^n\Par^-_0,$ and let $\Par_n \doteq \Par^+_n\cup \Par^-_n=z^n\Par_0.$ Then $\{\Par^+_n\}_{n\in \mathbb Z}$ and $\{\Par^-_n\}_{n\in \mathbb Z}$ are respectively the connected components of the space of positive and negative parabolic elements in $\univcover$; and two parabolic elements in $\univcover$ are conjugate if and only if they lie in the same connected component $\Par^+_n$ or $\Par^-_n$ for some $n\in\mathbb Z.$  For $n\in \mathbb{Z}$, we denote by $\overline{\Hyp_n}$ the closure of $\Hyp_n$ in $\univcover$, which is a union of 
    $\Hyp_n,$ $\Par_n$ and $\{z^n\}$. 
        
Unlike hyperbolic and parabolic elements, all elliptic elements of $\univcover$ lie in one-parameter subgroups; and we need to index their connected components differently. For $n>0$, let $\Ell_n$ be the subspace of  $\univcover$  consisting of elements conjugate to $\widetilde{\exp}\begin{bmatrix}
        0 & \theta \\
       -\theta & 0 
    \end{bmatrix}$ for some $\theta$ in $((n-1)\pi, n\pi),$ and for $n<0$, let  $\Ell_n$ be the subspace of  $\univcover$ consisting of elements conjugate to $\widetilde{\exp}\begin{bmatrix}
        0 & \theta \\
       -\theta & 0 
    \end{bmatrix}$ for some $\theta$ in $(n\pi, (n+1)\pi).$
    Then $\{\Ell_n\}_{n\in \mathbb Z\setminus\{0\}}$ are the connected components of the space of elliptic  elements in $\univcover;$ and two elliptic elements in $\univcover$ are conjugate if and only if both lie in the same connected component $\Ell_n$ and have the same trace. 
    For $n_1, n_2\in \mathbb{Z}$ with $n_1<n_2$, 
    $$\Ell_{n_2} = z^{n_2-n_1}\Ell_{n_1}$$ if $n_1$ and $n_2$ have the same sign, while 
    $$\Ell_{n_2} = z^{n_2-n_1-1}\Ell_{n_1}$$ if $n_1$ and $n_2$
    have opposite signs, \textcolor{black}{(because $\Ell_0$ is undefined).}
    See Figure \ref{fig: Universal_cover_Ell_colored2}.

\begin{figure}[hbt!]
        \centering
        \begin{overpic}[width=1.0\textwidth]{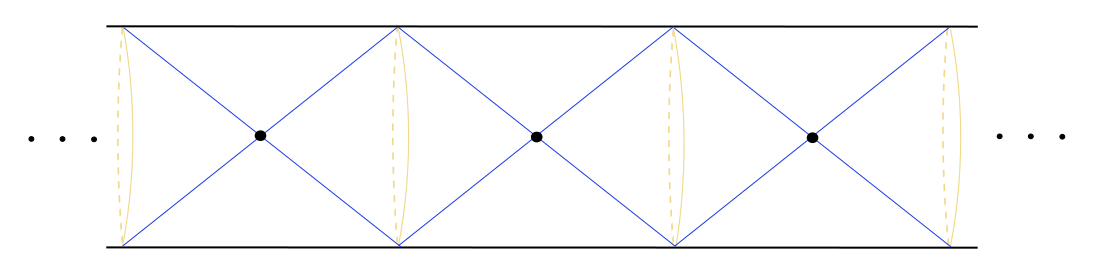}

            \put(22.0,13.5){\textcolor{black}{$z^{-1}$}}

            \put(48,13.5){\textcolor{black}{$\mathrm{I}$}}

            \put(72.5,13.5){\textcolor{black}{$z$}}

        \put(21,3){\textcolor{OliveGreen}{$\Hyp_{-1}$}}

            \put(46,3){\textcolor{OliveGreen}{$\Hyp_{0}$}}
            
            \put(71,3){\textcolor{OliveGreen}{$\Hyp_{1}$}}

            \put(14,20){\textcolor{blue}{$\Par^-_{-1}$}}

            \put(27,20){\textcolor{blue}{$\Par^+_{-1}$}}

            \put(39,20){\textcolor{blue}{$\Par^-_0$}}

            \put(53,20){\textcolor{blue}{$\Par^+_0$}}

            \put(64,20){\textcolor{blue}{$\Par^-_1$}}

            \put(78,20){\textcolor{blue}{$\Par^+_1$}}

        \put(10,11){\textcolor{red}{$\Ell_{-2}$}}
        
        \put(33,11){\textcolor{red}{$\Ell_{-1}$}}

        \put(59,11){\textcolor{red}{$\Ell_1$}}

        \put(84,11){\textcolor{red}{$\Ell_2$}}
        \end{overpic}
        \caption{\label{fig: Universal_cover_Ell_colored2} The universal covering $\univcover$}
    \end{figure}

     For a representation $\phi$ in $\HP(\Sigma)$, its relative Euler class is defined as follows.
    Let 
    $$\fund = \big\langle a_1, b_1, \cdots, a_g, b_g, c_1, \cdots, c_p \big|
    [a_1, b_1] \cdots [a_g, b_g] c_1\cdots c_p
    \big\rangle$$ be a presentation of $\pi_1(\Sigma)$, where $c_1,\cdots, c_p$ are \textcolor{black}{the} primitive peripheral elements of $\pi_1(\Sigma)$.  
    For each $i\in\{1,\dots, p\},$ as the image $\phi(c_i)$ of the peripheral element $c_i$ is either hyperbolic or parabolic, there exists a unique lift 
    $\widetilde{\phi(c_i)}$  of $\phi(c_i)$ in $\univcover$ that lies in $\overline{\Hyp_0}$.  
 For $j \in \{1,\cdots, g\}$, choose arbitrary lifts 
    $\widetilde{\phi(a_j)}$ and $\widetilde{\phi(b_j)}$ of $\phi(a_j)$ and $\phi(b_j)$ 
    in $\univcover$; and 
    note that the commutator 
    $[\widetilde{\phi(a_j)}, \widetilde{\phi(b_j)}]$ is independent of the choice of the lifts. 
    Since $$[\phi(a_1), \phi(b_1)] \cdots [\phi(a_g), \phi(b_g)] \phi(c_1)\cdots\phi(c_p)=\pm\mathrm{I},$$ 
    the product $[\widetilde{\phi(a_1)}, \widetilde{\phi(b_1)}] \cdots [\widetilde{\phi(a_g)}, \widetilde{\phi(b_g)}] \widetilde{\phi(c_1)}\cdots\widetilde{\phi(c_p)}$ in $\univcover$ projects to 
    $\pm\mathrm{I}$ in $\psl$ under the covering map, hence lies in the center of $\univcover.$ Hence, there exists an $n\in\mathbb Z$ such that $$[\widetilde{\phi(a_1)}, \widetilde{\phi(b_1)}] \cdots [\widetilde{\phi(a_g)}, \widetilde{\phi(b_g)}] \widetilde{\phi(c_1)}\cdots\widetilde{\phi(c_p)} = z^n.$$
   The \emph{relative Euler class} $e(\phi)$ of the representation $\phi\in\mathrm{HP}(\Sigma)$ is defined as
   $e(\phi)\doteq n.$
    \\
    
    The relative Euler class satisfies the following additivity property.
    \begin{proposition}\label{prop_additivity}\cite[Proposition 3.7]{goldman}
        Let $\Sigma = \Sigma_{g,p},$ and let $\gamma$ be a simple closed curve on $\Sigma$ separating $\Sigma$ into two subsurfaces $\Sigma_1$ and $\Sigma_2$.
     Suppose a representation $\phi\in \HP(\Sigma)$ maps $[\gamma]\in \pi_1(\Sigma)$ to a hyperbolic or parabolic element of $\psl.$ Then $\phi|_{\pi_1(\Sigma_1)}\in \mathrm{HP}(\Sigma_1)$, $\phi|_{\pi_1(\Sigma_2)}\in \mathrm{HP}(\Sigma_2)$, and 
        $$e(\phi) = e\big(\phi|_{\pi_1(\Sigma_1)}\big) + e(\phi|_{\pi_1(\Sigma_2)}).$$
    \end{proposition}

    Recall that a \emph{holonomy representation} is a representation $\phi\in \HP(\Sigma)$ that is Fuchsian, i.e., discrete and faithful, and that the convex core \textbf{core}$\big(\mathbb{H}^2/\phi\big(\pi_1(\Sigma)\big)\big)$ is homeomorphic to $\Sigma$. 
    The following proposition determines the relative Euler classes of holonomy representations. 
    In the original statement, $\phi$ is assumed to send all peripheral elements to hyperbolic elements of $\psl$; however, since type-preserving representations of $\pi_1(\Sigma)$ are holonomy if and only if they are Fuchsian (see \cite{yang}, Appendix A), the proof works verbatim here.
    \begin{proposition}\label{prop_holonomy}\cite[Proposition 3.4]{goldman}
        Let $\phi$ be a representation in $\HP(\Sigma)$. Then its relative Euler class $e(\phi)$ satisfies $|e(\phi)| = -\chi(\Sigma)$ if and only if $\phi$ is a holonomy representation.
    \end{proposition}

    \subsection{The signs}\label{sign}
    In this subsection, we recall the definition of the sign, and state the key result in \cite{RY}.
    Let $\Sigma = \Sigma_{g,p}$, and let $c_1, \cdots, c_p$ be the primitive peripheral elements of $\fund$, chosen for each puncture. The \emph{sign} of a representation $\phi\in\HP(\Sigma)$ is a $p$-tuple $s(\phi)=(s_1,\dots,s_p)$ in $\{-1,0,+1\}^p$, where for each $i\in\{1,\dots,p\},$ 
    \begin{equation*}
   s_i = \left\{
    \begin{array}{rl}
       +1,  &  \text{if } \phi(c_i) \text{ is positive parabolic}\\
        0, & \text{if } \phi(c_i) \text{ is hyperbolic}\\ 
        -1, &  \text{if } \phi(c_i) \text{ is negative parabolic}.
    \end{array}\right.
    \end{equation*}
    Since the type of an element of $\psl$ is invariant under $\psl-$conjugation, the sign of $\phi$ is independent on the choice of the primitive peripheral elements $c_i$'s. In the rest of the paper, for a sign $s\in \{-1,0,+1\}^p$, we respectively let  $p_+(s),$ $p_0(s)$ and $p_-(s)$ be the number of $+1$'s, $0$'s and $-1$'s in the components of  $s$.
    \smallskip

    The following proposition describes the behavior of the relative Euler class and the sign under the $\pgl\setminus\psl$-conjugations.

    \begin{proposition}\label{prop_pglpsl}\cite[Proposition 2.14]{RY}
    Let $h$ be an element of $\pgl\setminus \psl$; and for a representation $\phi\in\HP(\Sigma),$ let $h \phi h^{-1}:\pi_1(\Sigma)\to\psl$ be  defined by 
    $h \phi h^{-1}(c)=h \phi(c)h ^{-1}$
 for each $c\in\pi_1(\Sigma).$  Then $h \phi h^{-1}\in\HP(\Sigma)$ with the relative Euler class $e(h \phi h^{-1})=-e(\phi)$ and the sign $s(h \phi h^{-1})=-s(\phi).$    
    \end{proposition}
    
    Let $e: \HP(\Sigma)\to \mathbb{Z}$ be the relative Euler class map and let $s: \HP(\Sigma)\to \{-1,0,+1\}^p$ be the sign map defined respectively by sending $\phi$ to the relative Euler class $e(\phi)$ and the sign $s(\phi)$ of $\phi.$
    For $n\in \mathbb{Z}$ and $s\in \{-1,0,+1\}^p$, we let $\HPns(\Sigma) = e^{-1}(n)\cap s^{-1}(s)$ be the subspace of $\HP(\Sigma)$ consisting of representations with relative Euler class $n$ and sign $s.$ In particular, if $s\in \{\pm 1\}$, we denote this space by $\Rns$. 
    The following main result from \cite{RY} gives the necessary and sufficient condition for $\HPns(\Sigma)$ to be non-empty.

    \begin{theorem}\label{thm_exist_condition}\cite[Theorem 1.8]{RY}
        Let $\Sigma = \Sigma_{g,p}$, $n\in \mathbb{Z}$ and 
        $s \in \{-1,0,+1\}^p$ with $p_0(s)\geqslant 1$.
        Then $\HPns(\Sigma)$ is non-empty if and only if 
        $$\chi(\Sigma) + p_+(s)\leqslant n\leqslant -\chi(\Sigma) - p_-(s).$$
    \end{theorem}

    \subsection{Images of commutator and product maps}\label{ev}
    In this subsection, we describe the images of the commutator and product maps in $\univcover$, which will be used to compute the relative Euler classes.
    For a surface $\Sigma$ with $\chi(\Sigma) = -1$, its fundamental group $\pi_1(\Sigma)\cong \mathbb{F}_2$ is a free group of rank 2, 
    hence a representation $\phi:\pi_1(\Sigma)\to \psl$ is determined by the $\psl$-images of the two generators. 
    As a consequence, 
    $$\mathrm{Hom}\big(\pi_1(\Sigma),  \psl\big)\cong \psl\times\psl.$$
    If $\Sigma = \Sigma_{1,1}$ with $\pi_1(\Sigma) = \langle a_1, b_1, c_1\ |\ [a_1,b_1]\cdot c_1\rangle$, then for $\phi\in \HP(\Sigma)$, we let $\widetilde{\phi(a_1)}, \widetilde{\phi(b_1)}$ be arbitrary lifts of $\phi(a_1), \phi(b_1)$ in $\univcover$, and let $\widetilde{\phi(c_1)}$ be the unique lift of $\phi(c_1)$ in $\overline{\Hyp_0}$. If $e(\phi) =n$, then by the definition of the relative Euler class, we have
    $$[\widetilde{\phi(a_1)},
    \widetilde{\phi(b_1)}]\cdot 
    \widetilde{\phi(c_1)}
    = z^n,$$ hence the commutator $$\big[\widetilde{\phi(a_1)},
    \widetilde{\phi(b_1)}\big]
    = z^n \widetilde{\phi(c_1)}^{-1}
    $$ lies in $\overline{\Hyp_n}=
    z^n\overline{\Hyp_0}$. 
    Similarly, if $\Sigma = \Sigma_{0,3}$ with $\pi_1(\Sigma) = \langle c_1, c_2, c_3\ |\ c_1c_2c_3 \rangle$, then for $\phi\in \HP(\Sigma)$, we let $\widetilde{\phi(c_1)}, \widetilde{\phi(c_2)}, \widetilde{\phi(c_3)}$ be the unique lifts of $\phi(c_1), \phi(c_2), \phi(c_3)$ in $\overline{\Hyp_0}$; and if $e(\phi) =n$, then we have
    $$\widetilde{\phi(c_1)}\widetilde{\phi(c_2)}\widetilde{\phi(c_3)}
    = z^n,$$ and hence the product $$\widetilde{\phi(c_1)}\widetilde{\phi(c_2)}
    = z^n \widetilde{\phi(c_3)}^{-1}
    $$ lies in $\overline{\Hyp_n}$. 
    \\
    
    The following Theorem \ref{thm_liftcommu} and Theorem \ref{thm_prodimage} respectively describe the images of the commutator and product maps in $\univcover$. See Figure \ref{fig: evimage}.
    \begin{theorem}\cite[Theorem 7.1]{goldman}\label{thm_liftcommu}
            Let $\widetilde{R}: \univcover\times \univcover\to \univcover$ be the commutator map sending $(\widetilde{A}, \widetilde{B})$ to $[\widetilde{A},\widetilde{B}]
            = \widetilde{A}\widetilde{B}\widetilde{A}^{-1}\widetilde{B}^{-1}$. 
            The image $\mathrm{Im}\widetilde{R}$ of 
            $\widetilde{R}$ equals $$\mathcal{I} = 
            \mathrm{Hyp}_{- 1}
            \cup 
            \mathrm{Par}_{-1}^{+}
            \cup
            \mathrm{Ell}_{-1}
            \cup
            \mathrm{Par}_0
            \cup 
            \{\mathrm I\}
            \cup 
            \mathrm{Hyp}_0
            \cup
            \mathrm{Ell}_{1}
            \cup 
            \mathrm{Par}_{1}^{-}
            \cup
            \mathrm{Hyp}_{1}.
            $$
        \end{theorem}
        \begin{theorem}\label{thm_prodimage}\cite[Proposition 3.4, Theorem 3.7, and Lemma 7.5]{RY} 
        Let $m: \univcover\times \univcover\to \univcover$ be the product map sending $(\widetilde{A}, \widetilde{B})$ to $\widetilde{A}\widetilde{B}$, and let $s, s_1, s_2\in \{\pm 1\}$.\\
        For hyperbolic image under $m$, the following holds:
        \begin{enumerate}[(1)]

            \item $m(\Hyp_0\times \Hyp_0)\cap \big(\bigcup_{k\in \mathbb{Z}}\Hyp_k\big)
            = \Hyp_{-1}\cup  \Hyp_0\cup\Hyp_1$.

            \item $m(\Par_0^{sgn(s)}\times \Hyp_0)\cap \big(\bigcup_{k\in \mathbb{Z}}\Hyp_k\big) = \Hyp_0\cup \Hyp_{s}$.

            \item $m\big(\Par_0^{sgn(s_1)}\times \Par_0^{sgn(s_2)}\big)\cap \big(\bigcup_{k\in \mathbb{Z}}\Hyp_k\big) = \Hyp_{\frac{s_1+s_2}{2}}$.

        \end{enumerate}
        For elliptic image under $m$, the following holds:
        \begin{enumerate}[(1)] 


            \item $m\big(\Par_0^{sgn(s)} \times\Par_0^{sgn(s)} \big)\cap \big(\bigcup_{k\in \mathbb{Z}\setminus \{0\}}\Ell_k\big) = \Ell_s$.

            \item $m\big(\Par_0\times\Ell_1 \big)\cap \big(\bigcup_{k\in \mathbb{Z}\setminus \{0\}}\Ell_k\big) = \Ell_1$.

            \item 
            $m\big(\Hyp_0 \times\Hyp_0\big)\cap \big(\bigcup_{k\in \mathbb{Z}\setminus \{0\}}\Ell_k\big) \subset \Ell_{-1}\cup \Ell_1$.
        \end{enumerate}
    \end{theorem}

    \begin{figure}[hbt!]
        \centering
        \begin{overpic}[width=0.9\textwidth]{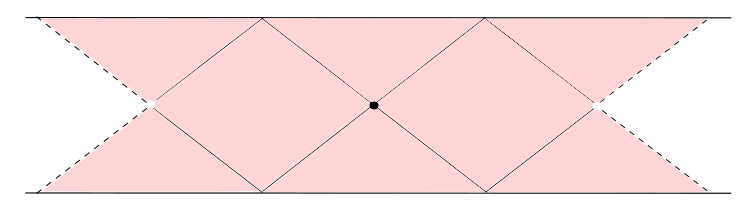}


            \put(49,15){\textcolor{black}{$\mathrm{I}$}}


        \put(17,4){\textcolor{OliveGreen}{$\Hyp_{-1}$}}

            \put(47,4){\textcolor{OliveGreen}{$\Hyp_{0}$}}
            
            \put(77,4){\textcolor{OliveGreen}{$\Hyp_{1}$}}


            \put(24,23){\textcolor{blue}{$\Par^+_{-1}$}}

            \put(39,23){\textcolor{blue}{$\Par^-_0$}}

            \put(55,23){\textcolor{blue}{$\Par^+_0$}}

            \put(68,23){\textcolor{blue}{$\Par^-_1$}}


        
        \put(32.5,13){\textcolor{red}{$\Ell_{-1}$}}

        \put(62.5,13){\textcolor{red}{$\Ell_1$}}

        \end{overpic}
        \caption{\label{fig: evimage} The image of the commutator map $\widetilde{R}: \univcover\times \univcover\to \univcover$}
    \end{figure}

\section{Proof of the main results}\label{results}
    Let $\Sigma = \Sigma_{g,p}$. By abuse of notation, we use the term ``simple closed curve" to refer both to the actual curve on $\Sigma$ and to its homotopy class in $\pi_1(\Sigma)$. Note that an oriented simple closed curve $\gamma$ on $\Sigma$ determines a conjugacy class in $\pi_1(\Sigma)$, hence its image under a representation to $\psl$ is well-defined up to conjugation. Moreover, we call an element of $\pi_1(\Sigma)$ \emph{separating} (resp. \emph{non-separating}) if it is represented by a separating (resp. non-separating) simple closed curve on $\Sigma$.

    Recall that a type-preserving representation $\phi$ is totally hyperbolic if it sends every non-peripheral simple closed curve in $\pi_1(\Sigma)$ to a hyperbolic element of $\psl$. 
    Since conjugation by an element of $\psl$ preserves hyperbolic elements, the total hyperbolicity of $\phi$ is independent of the choice of representatives for the simple closed curves. Moreover,
    conjugation by any element of $\pgl$, including those in $\pgl\setminus \psl$, preserves hyperbolic elements; and hence $\phi$ is totally hyperbolic if and only if any of its $\pgl$-conjugates is totally hyperbolic.
    Therefore, by Proposition \ref{prop_pglpsl}, 
    it suffices to show Theorem \ref{thm_counter_Bowditch} for relative Euler class $-\chi(\Sigma) - 1$.
    \medskip

    For $n \in \mathbb{Z}\setminus \{0\}$ and $s\in \{\pm 1\}^p$, define $\mathcal{M}^s_n(\Sigma)\doteq \Rns\slash \psl$ where the quotient is taken by \textcolor{black}{the} $\psl$-conjugation. 
    \textcolor{black}{(Note that for $n\neq 0$, the GIT quotient coincides with the quotient by the $\psl$-conjugation.)}
    Then the space $\mathcal{M}^s_n(\Sigma)$ admits a measure induced by \textcolor{black}{the} Goldman symplectic form, which coincides with the Weil-Petersson symplectic form when $|n| = -\chi(\Sigma)$ \cite{goldman_measure}.
    Let $\NPns\subset \Rns$ be a set of representations sending no non-peripheral simple closed curve to a parabolic element or to $\pm \mathrm I$. Notice that $\phi\in \NPns$ implies that \textcolor{black}{all the
    $\psl$-conjugations} of $\phi$ are in $\NPns$, hence the quotient $\NPns/\psl$ is a well-defined subspace of $\mathcal{M}^s_n(\Sigma)$.
    \medskip
    
    Theorem \ref{thm_counter_Bowditch} for relative Euler class $-\chi(\Sigma) - 1$ follows from Proposition \ref{prop_nonpara} and Theorem \ref{thm_tothyp}.

    \begin{proposition}\label{prop_nonpara}
        Let $\Sigma = \Sigma_{g,p}$ with $\chi(\Sigma)\leqslant -2$ and $p\geqslant 1$, and let $n = -\chi(\Sigma) - 1$ and $s\in \{\pm 1\}^p$ with $p_-(s) = 1$. 
        Then the space $\NPns/\psl$ has full measure in $\mathcal{M}^s_n(\Sigma)$. As a consequence, $\NPns$ contains uncountably many representations.
    \end{proposition}

\begin{theorem}\label{thm_tothyp}
        Let $\Sigma = \Sigma_{g,p}$ with $\chi(\Sigma)\leqslant -2$ and $p\geqslant 1$,  and let $n = -\chi(\Sigma) - 1$ and $s\in \{\pm 1\}^p$ with $p_-(s) = 1$. 
        Then every representation $\phi\in \NPns$ is totally hyperbolic.
    \end{theorem}

    The outline of this section is as follows. 
    In Subsection \ref{proof_prop} and Subsection \ref{proof_thm}, we will prove
    Proposition \ref{prop_nonpara} and Theorem \ref{thm_tothyp}, respectively. 
    Then in Subsection \ref{proof_cor}, we will prove Corollary \ref{cor_almost_fuchsian}.

\subsection{Proof of Proposition \ref{prop_nonpara}: $\NPns$ has full measure in $\Rns$}\label{proof_prop}

    To prove Proposition \ref{prop_nonpara}, we need the following Lemma \ref{lem_nonpara}.
    \begin{lemma}\label{lem_nonpara}
        Let $\Sigma = \Sigma_{g,p}$ 
        with $\chi(\Sigma)\leqslant -1$ and $p\geqslant 1$. Let $c_1, c_2,\cdots, c_p$ be \textcolor{black}{the} primitive peripheral elements of $\pi_1(\Sigma)$, and let $\pm C$ be a hyperbolic element of $ \psl$.
        Then there exists a representation $\phi\in \HP(\Sigma)$ of relative Euler class $e(\phi) = -\chi(\Sigma)$, that sends $c_p$ to $\pm C$ and
        all other primitive peripheral elements to positive parabolic elements.
    \end{lemma}
    \begin{proof}
        We will proceed by induction
    on the Euler characteristic $\chi(\Sigma)$.
    \smallskip

    As the base case of the induction, we assume $\chi(\Sigma) = -1$, where $\Sigma$ is either a one-holed torus $\Sigma_{1,1}$ or 
    a three-holed sphere $\Sigma_{0,3}$. 
    Let $\widetilde{C}$ be the lift of $\pm C$ in $\Hyp_1$. 
    If $\Sigma\cong \Sigma_{1,1}$ with $\pi_1(\Sigma) = \langle a, b, c_1 : [a,b] = c_1 \rangle$, then
    by Proposition \ref{thm_liftcommu}, there exists a pair $(\widetilde{A}, \widetilde{B})\in \univcover\times\univcover$ whose commutator $\big[\widetilde{A}, \widetilde{B}\big]$ equals $\widetilde{C}$; and we can define $\phi$ by letting $\phi(a)$ and $\phi(b)$ be the projections of $\widetilde{A}$ and $\widetilde{B}$, respectively, to $\psl$. 
    If $\Sigma\cong \Sigma_{0,3}$ with $\pi_1(\Sigma) = \langle c_1, c_2, c_3 : c_1 c_2 = c_3 \rangle$, 
    then by Theorem \ref{thm_prodimage}, there exists a pair $(\widetilde{A}, \widetilde{B})\in \Par_0^+ \times\Par_0^+$ whose product $\widetilde{A} \widetilde{B}$  equals $\widetilde{C}$; and we can define $\phi$ by letting $\phi(c_1)$ and $\phi(c_2)$ be the projections of $\widetilde{A}$ and $\widetilde{B}$, respectively, to $\psl$. 
    In both cases, $\phi$ sends $c_p$ to the projection of $\widetilde{C}$ to $\psl$, which is $\pm C$.
    Moreover, as $\widetilde{C}\in \Hyp_1$, we have $e(\phi) = 1$; 
    and $\phi$ sends all other primitive peripheral elements of $\pi_1(\Sigma)$ to positive parabolic elements, as desired.
    \smallskip

    Now let $\chi(\Sigma)\leqslant -2$, and assume that the lemma holds for any subsurface of $\Sigma$ with Euler characteristic greater than $\chi(\Sigma)$. Let $\widetilde{C}$ be the lift of $\pm C$ in $\Hyp_1$. 
    We will first choose a pair of pants $P\subset \Sigma$ where $c_p\in \pi_1(P)$ and construct $\phi$ on $\pi_1(P)$, for the cases $p\geqslant 2$ and $p = 1$ separately, then extend $\phi$ to $\pi_1(\Sigma)$.
    
    If $p \geqslant 2$, then we choose $P\subset \Sigma$ such that the complement $\Sigma\setminus P$ is connected. Then 
    $\pi_1(P) = \langle d, c_{p-1}, c_p : d  c_{p-1} = c_p \rangle$, where $d$ is represented by the common boundary component of $P$ and $\Sigma\setminus P$. 
    By Theorem \ref{thm_prodimage}, there exists a pair $(\widetilde{A}, \widetilde{B})\in  \Hyp_0\times\Par^+_0$ whose product $\widetilde{A} \widetilde{B}$  equals $\widetilde{C}$; and we can define $\phi$ on $\pi_1(P)$ by letting $\phi(d)$ and $\phi(c_{p-1})$ be the projections of $\widetilde{A}$ and $\widetilde{B}$, respectively, to $\psl$. 
    
    If $p = 1$, then we choose $P\subset \Sigma$ such that the complement $\Sigma\setminus P$ has two connected components $\Sigma_1$ and $\Sigma_2$. Then 
    $\pi_1(P) = \langle d_1, d_2, c_1 : d_1 d_2 = c_1 \rangle$, where $d_i$ is represented by the common boundary component of $P$ and $\Sigma_i$ for $i\in \{1,2\}$. 
    By Theorem \ref{thm_prodimage}, there exists a pair $(\widetilde{A}, \widetilde{B})\in \Hyp_0 \times\Hyp_0$ whose product $\widetilde{A} \widetilde{B}$  equals $\widetilde{C}$; and we can define $\phi$ on $\pi_1(P)$ by letting $\phi(d_1)$ and $\phi(d_2)$ be the projections of $\widetilde{A}$ and $\widetilde{B}$, respectively, to $\psl$. 
    
    Since each connected component of $\Sigma\setminus P$ has Euler characteristic greater than $\chi(\Sigma)$, and since
    $\phi(d), \phi(d_1),$ and $ \phi(d_2)$ are hyperbolic, we can extend $\phi$ to $\pi_1(\Sigma)$ by defining  $\phi|_{\pi_1(\Sigma\setminus P)}$ using the induction hypothesis. 
    Since $\widetilde{C}\in \Hyp_1$, we have $e\big(\phi|_{\pi_1(P)}\big) = 1$; and by Proposition \ref{prop_additivity}, we have $e(\phi) = e\big(\phi|_{\pi_1(P)}\big)+ e\big(\phi|_{\pi_1(\Sigma\setminus P)}\big) = 1 -\chi(\Sigma\setminus P)= -\chi(\Sigma)$. Moreover, by the definition of $\phi|_{\pi_1(P)}$ and the induction hypothesis, 
    we have $\phi(c_p) = \pm C$ and $\phi(c_i)\in \Par^+$ for $i\in \{1,\cdots, p-1\}$.
    \end{proof}

    We now prove Proposition \ref{prop_nonpara}.
    \begin{proof}[Proof of Proposition \ref{prop_nonpara}]
        For a non-peripheral simple closed curve $\gamma \in \pi_1(\Sigma)$, we let $\mathrm{P}_\gamma\doteq \{[\phi]\in \mathcal{M}^s_n(\Sigma) \big|\ \big|\tr\big(\widetilde{\phi(\gamma)}\big)\big| = 2\}$, 
        where $\tr\big(\widetilde{\phi(\gamma)}\big)$ is the trace of a lift $\widetilde{\phi(\gamma)}$ of $\phi(\gamma)$ in $\SL$.
        Let $\mathrm P$ be the union of $\mathrm{P}_\gamma$ for all non-peripheral simple closed curves $\gamma \in \pi_1(\Sigma)$. Notice that $\NPns/\psl$ is the complement of $\mathrm P$ in $\mathcal{M}^s_n(\Sigma)$.
        We will show that for each $\gamma$, $\mathrm{P}_\gamma$ has measure zero in $\mathcal{M}^s_n(\Sigma)$, hence their countable union $\mathrm P$ has measure zero in $\mathcal{M}^s_n(\Sigma)$. 
        \medskip
        
        To this end, for each $\gamma$, we will construct $\phi\in \Rns$ such that 
        $\phi(\gamma)$ is hyperbolic. As
        $\big|\tr\big(\widetilde{\phi(\gamma)}\big)\big| > 2$, this shows that the complement $\mathcal{M}^s_n(\Sigma)\setminus \mathrm{P}_\gamma$ is non-empty, concluding that the zero set
        $\mathrm{P}_\gamma$ of an analytic function $\tr(\phi(\gamma))^2 - 4$ has measure zero \cite{Mityagin}. 
        \smallskip
     
        If $\gamma$ is non-separating, 
        then 
        there is a one-holed torus $T\subset \Sigma$ containing $\gamma$, which shares a common boundary $d$ with its complement $\Sigma\setminus T$. 
        Let $c_1,\cdots, c_p$ be \textcolor{black}{the} primitive peripheral elements of $\pi_1(\Sigma)$, chosen for each puncture; and observe that $d$, together with $c_1,\cdots, c_p$, represent the $p+1$ punctures of $\Sigma\setminus T$. 
        Define $s'\in \{-1,0,1\}^{p+1}$ by letting $s'_i=s_i$ for $i\in\{1,\dots, p\}$ and $s'_{p+1} = 0$. 
        Then $p_-(s) = 1$ implies $p_-(s') = 1$, and we have $n-1 = -\chi(\Sigma) - 2 =  -\chi(\Sigma\setminus T) - p_-(s')$.
        By Theorem \ref{thm_exist_condition}, there exists a $\phi\in \HP^{s'}_{n-1}(\Sigma\setminus T)$ that sends $d$ to a hyperbolic element. Applying Lemma \ref{lem_nonpara} to $T$ with $\pm C = \phi(d)$, we can extend $\phi$ to $\pi_1(\Sigma)$
        so that the relative Euler class $e\big(\phi|_{\pi_1(T)}\big) = 1$. 
        By Proposition \ref{prop_additivity}, $e(\phi) = (n-1) + 1 = n$; and $s(\phi) = s$, i.e., $\phi\in \Rns$. 
        Moreover, by Proposition \ref{prop_holonomy}, $\phi|_{\pi_1(T)}$ is a holonomy representation, hence $\phi(\gamma
        )$ is hyperbolic.
        \smallskip

        If $\gamma$ is separating, 
        then it separates $\Sigma$ into two subsurfaces 
        $\Sigma_1$ and $\Sigma_2$. 
        Up to a permutation of the peripheral elements, we can assume that $s_1=-1$ and $s_2 = \cdots=s_{p}= +1$. After possibly reindexing, we may further assume that $c_1,\cdots, c_k\in \pi_1(\Sigma_1)$ and $c_{k+1},\cdots, c_p\in \pi_1(\Sigma_2)$ for some $k\in \{1,\cdots, p\}$. 
        Observe that $\gamma$, together with $c_1,\cdots, c_k$, represent the $k+1$ punctures of $\Sigma_1$. Define $s'\in \{-1,0,1\}^{k+1}$ by letting $s'_i = s_i$ for $i\in \{1,\cdots, k\}$, and $s'_{k+1} = 0$.
        By Theorem \ref{thm_exist_condition}, there exists a $\phi\in \HP^{s'}_{-\chi(\Sigma_1)-1}(\Sigma_1)$, where $\phi(c_1)$ is negative parabolic, and $\phi(\gamma)$ is hyperbolic. 
    Applying Lemma \ref{lem_nonpara} to $\Sigma_2$ with $\pm C = \phi(\gamma)$, we can extend $\phi$ to $\pi_1(\Sigma)$
    so that the relative Euler class $e\big(\phi|_{\pi_1(\Sigma_2)}\big) = -\chi(\Sigma_2)$. 
    By Proposition \ref{prop_additivity}, we have $e(\phi) = 
    \big(-\chi(\Sigma_1) - 1\big) + \big(-\chi(\Sigma_2)\big) = n$; and since $\phi(c_1)\in \Par^-$ and $\phi(c_i)\in \Par^+$ for all $i\in \{2,\cdots, p\}$, $s(\phi) = s$, i.e., $\phi\in \Rns$. 
    \end{proof}

\subsection{Proof of Theorem \ref{thm_tothyp}: total hyperbolicity of representations in $\NPns$}\label{proof_thm}

    We first state
    Lemma \ref{lem_offdiag} and Lemma \ref{lem_offdiag_Ell}, which are needed for the proof of Theorem \ref{thm_tothyp}. 
    The proofs rely on the definitions of $\Par^{\pm}_n$ and $\Ell_n$;
    see \cite{RY} for details.
  \begin{lemma}\label{lem_offdiag}\cite[Lemma 2.13]{RY} 
        For $n\in \mathbb{Z}$ and a parabolic element $\widetilde{A}\in \mathrm{Par}_n$ of $\univcover,$ let  $s(\widetilde{A})\in\{\pm\}$ be its sign, i.e., $s(\widetilde{A})=+$ if $\widetilde{A}\in \mathrm{Par}^+_n,$ and $s(\widetilde{A})=-$ if $\widetilde{A}\in \mathrm{Par}^-_n.$ Let $A$ be the projection of $\widetilde{A}$ to $\SL$, and for $i,j\in\{1,2\}$ let $a_{ij}$ be the $(i,j)$-entry of $A.$ Then either $a_{12}\neq0$ or $a_{21}\neq 0.$  Moreover:
        
        \begin{enumerate}[(1)]
       \item If $n$ is even, then $s(\widetilde{A})=sgn(a_{12})$ if $a_{12}\neq 0,$ and $s(\widetilde{A})=-sgn(a_{21})$ if $a_{21}\neq 0.$
       
    \item If $n$ is odd, then $s(\widetilde{A})=-sgn(a_{12})$ if $a_{12}\neq 0,$  and $s(\widetilde{A})=sgn(a_{21})$ if $a_{21}\neq 0.$
    \end{enumerate}
    \end{lemma}

    \begin{lemma}\label{lem_offdiag_Ell} \cite[Lemma 7.4]{RY}
        For $n\in \mathbb{Z}$ and an elliptic element $\widetilde{A}\in \mathrm{Ell}_n$ of $\univcover,$ let $A$ be its projection to $\SL$, and for $i,j\in\{1,2\}$ let $a_{ij}$ be the $(i,j)$-entry of $A.$ 
        Then $a_{12}\neq0$ and $a_{21}\neq 0.$  Moreover:
        \begin{enumerate}[(1)]
            \item 
            If $n$ is odd, then $sgn(n)=sgn(a_{12}) = -sgn(a_{21})$.
            
            \item 
            If $n$ is even, then $sgn(n)=-sgn(a_{12}) = sgn(a_{21})$.
        \end{enumerate}
    \end{lemma}

    The proof of Theorem \ref{thm_tothyp} will be based on Lemma \ref{lem_evimage_base} and Proposition \ref{prop_evimage} below. 
    See also Theorem \ref{thm_prodimage} for comparison.
    \begin{lemma}\label{lem_evimage_base} Let  $m: \univcover\times \univcover
    \to \univcover$ defined by $m(\widetilde{A}, \widetilde{B}) = \widetilde{A}\widetilde{B}$, and let  $s\in \{\pm 1\}$.
    \begin{enumerate}[(1)] 
    

    \item $m\big(\Hyp_0 \times\Par^{sgn(s)}_0\big)\cap \big(\bigcup_{k\in \mathbb{Z}\setminus \{0\}}\Ell_k\big) \subset \Ell_{s}$.
    

    \item $m\big(\Hyp_0 \times\Ell_1\big)\cap \big(\bigcup_{k\in \mathbb{Z}\setminus \{0\}}\Ell_k\big) \subset \Ell_1$. 

    \item $m\big(\Ell_{-1} \times\Ell_1\big)\cap \big(\bigcup_{k\in \mathbb{Z}\setminus \{0\}}\Ell_k\big) \subset \Ell_{-1}\cup \Ell_1$.
    \end{enumerate}
    \end{lemma}

    \begin{proof}
        
        For $\widetilde{A}, \widetilde{B}\in \Ell_{-1}\cup \overline{\Hyp_0}\cup \Ell_1$, recall that the product $m(\widetilde{A}, \widetilde{B}) = \widetilde{A}\widetilde{B}$ is defined as follows. 
        As $\widetilde{A}$ and $\widetilde{B}$ lie in $\Ell_{-1}\cup \overline{\Hyp_0}\cup \Ell_1$, 
        each of them generates a one-parameter subgroup of $\univcover$. 
        Let $\{\widetilde{A_t}\}\interval$ be a path connecting $\mathrm{I}$ and $\widetilde{A}$ in the one-parameter subgroup of $\univcover$ generated by $\widetilde{A},$ and let 
       $\{\widetilde{B_t}\}\interval$ be a path connecting $\mathrm{I}$ and $\widetilde{B}$ in the one-parameter subgroup of $\univcover$ generated by $\widetilde{B}.$ 
       Let 
       $\{A_t\}\interval$ and $\{B_t\}\interval$ respectively be the projections of $\{\widetilde{A_t}\}\interval$ and $\{\widetilde{B_t}\}\interval$ to $\SL$.
        Letting
        $
        \{
        \widetilde{A_tB_t}
        \}\interval
        $ 
        be the lift of the path $\{A_tB_t\}_{t\in [0,1]}$ of $\SL$ to
        $\univcover$ starting from $\mathrm{I}$, the product $\widetilde{A}\widetilde{B}$ is defined by its endpoint $\widetilde{A_1B_1}$.
        \medskip
        
        For (1), 
        we let $\big(\widetilde{A}, \widetilde{B}\big)\in \Hyp_0 \times\Par^{sgn(s)}_0$, and assume that their product $m\big(\widetilde{A}, \widetilde{B}\big) = \widetilde{A}\widetilde{B}$ is elliptic in $\univcover$.
        As $\widetilde{A}\in \Hyp_0$,
        up to an $\SL-$conjugation, we can assume that 
        $$A_1= \begin{bmatrix}
        e^\lambda & 0 \\
        0 & e^{-\lambda}
        \end{bmatrix}$$ for some $\lambda>0$; and by a re-parametrization if necessary, we can assume that 
        $$A_t= \begin{bmatrix}
        e^{\lambda t} & 0 \\
        0 & e^{-\lambda t}
        \end{bmatrix}$$
        for all $t\in [0,1]$.
        Letting $B_t = 
        \begin{bmatrix}
            a_t & b_t \\
            c_t & d_t
        \end{bmatrix}$ 
        with trace $\tr(B_t) = a_t+d_t = 2$
        for all $t\in [0,1],$
        their product equals to 
        $$A_tB_t = 
        \begin{bmatrix}
            a_te^{\lambda t} & b_te^{\lambda t} \\
            c_te^{-\lambda t} & d_te^{-\lambda t}
        \end{bmatrix}.$$
        \textcolor{black}{If $s = +1$, then 
        by Lemma \ref{lem_offdiag}, we have $b_t\geqslant 0$ and $c_t\leqslant 0$ for all $t\in [0,1]$, hence $(A_tB_t)_{12}  \geqslant 0$ and $(A_tB_t)_{21} \leqslant 0$.
        By Lemma \ref{lem_offdiag_Ell}, the path $\{\widetilde{A_tB_t}\}\interval$ intersects neither $\Ell_{-1}$ nor $\Ell_2$, hence lies in $\overline{\Hyp_0}\cup \overline{\Hyp_1}\cup \Ell_1$ as it is the connected component 
        of $\univcover\setminus \big(\Ell_{-1}\cup \Ell_2\big)$ that contains $\widetilde{A_0B_0} = \mathrm I$. (See Figure \ref{fig: Lemma_3.6}.)} As $\widetilde{A_1B_1} = \widetilde{A}\widetilde{B}$ is elliptic, we have $\widetilde{A_1B_1}\in \Ell_1$. 
        Similarly, 
        if $s = -1$, then
        by Lemma \ref{lem_offdiag}, we have $b_t\leqslant 0$ and $c_t\geqslant 0$, hence $(A_tB_t)_{12} = b_te^{\lambda t} \leqslant 0$ and $(A_tB_t)_{21} = c_te^{-\lambda t} \geqslant 0$ for all $t\in [0,1]$. 
        Then by Lemma \ref{lem_offdiag_Ell}, the path $\{\widetilde{A_tB_t}\}\interval$ intersects neither $\Ell_1$ nor $\Ell_{-2}$, hence lies in $\overline{\Hyp_{-1}}\cup \overline{\Hyp_0}\cup \Ell_{-1}$ as it is the connected component 
        of $\univcover\setminus \big(\Ell_1\cup \Ell_{-2}\big)$ that contains $\widetilde{A_0B_0} = \mathrm I$. 
        As $\widetilde{A_1B_1} = \widetilde{A}\widetilde{B}$ is  elliptic, we have $\widetilde{A_1B_1}\in \Ell_{-1}$.

        For (2), we let $\big(\widetilde{A}, \widetilde{B}\big)\in \Hyp_0 \times\Ell_1$, and assume that $m\big(\widetilde{A}, \widetilde{B}\big)$ is elliptic in $\univcover$.
        Similar to (1), 
        we can assume 
        $A_t= \begin{bmatrix}
        e^{\lambda t} & 0 \\
        0 & e^{-\lambda t}
        \end{bmatrix}$ and $B_t = 
        \begin{bmatrix}
            a_t & b_t \\
            c_t & d_t
        \end{bmatrix}$ 
        with trace $\tr(B_t) = a_t+d_t \in (-2,2)$
        for all $t\in [0,1].$
        By Lemma \ref{lem_offdiag_Ell}, $b_t\geqslant 0$ and $c_t\leqslant 0$ for all $t\in [0,1]$, and the rest of the proof follows verbatim the previous case $s = +1$ in (1). 

        \begin{figure}[hbt!]
        \centering
        \begin{overpic}[width=0.9\textwidth]{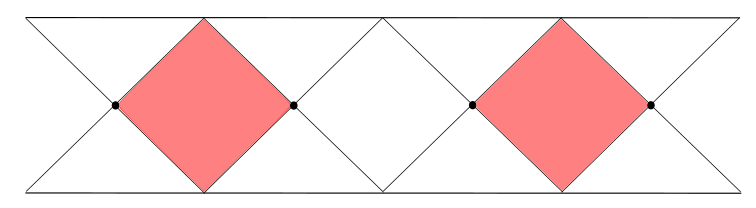}


            \put(14,15){\textcolor{black}{$z^{-1}$}}
            
            \put(38.5,15){\textcolor{black}{$\mathrm{I}$}}

            \put(62,15){\textcolor{black}{$z$}}

            \put(85,15){\textcolor{black}{$z^2$}}


        \put(13,3){\textcolor{black}{$\Hyp_{-1}$}}

            \put(37,3){\textcolor{black}{$\Hyp_{0}$}}
            
            \put(60,3){\textcolor{black}{$\Hyp_{1}$}}

            \put(84,3){\textcolor{black}{$\Hyp_{2}$}}


            \put(18,23){\textcolor{black}{$\Par^+_{-1}$}}
            
            \put(29,23){\textcolor{black}{$\Par^-_0$}}

            \put(42.5,23){\textcolor{black}{$\Par^+_0$}}

            \put(53,23){\textcolor{black}{$\Par^-_1$}}

            \put(66,23){\textcolor{black}{$\Par^+_{1}$}}

            \put(77,23){\textcolor{black}{$\Par^-_{2}$}}


        
        \put(25,13){\textcolor{black}{$\Ell_{-1}$}}

        \put(49,13){\textcolor{black}{$\Ell_{1}$}}

        \put(72.5,13){\textcolor{black}{$\Ell_2$}}

        \end{overpic}
        \caption{\label{fig: Lemma_3.6} $\Ell_{-1}$ and $\Ell_{2}$}
    \end{figure}

        For (3), we let $\big(\widetilde{A}, \widetilde{B}\big)\in \Ell_{-1} \times\Ell_1$, and assume that $m\big(\widetilde{A}, \widetilde{B}\big)$ is elliptic in $\univcover$.
        As $\widetilde{A}\in \Ell_{-1}$,
        up to an $\SL-$conjugation, we can assume that $$A_1=  \elltheta$$ for some $\theta\in (-\pi, 0)$; and by a re-parametrization if necessary, we can assume that 
        $$A_t= \begin{bmatrix}
        \cos(\theta t) & \sin(\theta t) \\
        -\sin(\theta t) & \cos(\theta t)
        \end{bmatrix}$$
        for all $t\in [0,1].$
        Letting
        $B_t = 
        \begin{bmatrix}
            a_t & b_t \\
            c_t & d_t
        \end{bmatrix}$ 
        for all $t\in [0,1],$
        their product equals to 
        $$A_tB_t = 
        \begin{bmatrix}
            a_t\cos(\theta t)+c_t\sin(\theta t) & b_t\cos(\theta t)+d_t\sin(\theta t) \\
            c_t\cos(\theta t) - a_t\sin(\theta t) & d_t\cos(\theta t) - b_t\sin(\theta t)
        \end{bmatrix},$$
        where the trace $\tr(A_tB_t) = (a_t + d_t)\cos(\theta t)+(c_t- b_t)\sin(\theta t)$. 
        Since $a_t + d_t = \tr(B_t)  \in (-2,2)$, we have $(a_t + d_t)\cos(\theta t)\in (-2,2)$. Moreover, 
        by Lemma \ref{lem_offdiag_Ell}, $b_t>0$  and $c_t<0$ for all $t\in (0,1]$, hence
        $c_t- b_t < 0$; and since $\sin(\theta t ) < 0$, we have $(c_t- b_t)\sin(\theta t ) > 0$. Therefore, the trace $\tr(A_tB_t) > -2$. This implies that 
        $\widetilde{A_tB_t}$ never passes $\Par_{-1}^+\cup \Par_1^-\cup \{z^{\pm 1}\}$, hence lies in $\Ell_{-1}\cup\overline{\Hyp_0}\cup \Ell_1$ as it is the connected component 
            of $\univcover\setminus \big(\Par_{-1}^+\cup \Par_1^-\cup \{z^{\pm 1}\}\big)$ that contains $\widetilde{A_0B_0} = \mathrm I$. 
            \textcolor{black}{(See Figure \ref{fig: Lemma_3.6}.)}           As $\widetilde{A_1B_1} = \widetilde{A}\widetilde{B}$ is elliptic, we have $\widetilde{A_1B_1}\in \Ell_{-1}\cup\Ell_1$. 
            \end{proof}

        When a surface $\Sigma$ has Euler characteristic $\chi(\Sigma) = -1$, 
        recall from Subsection \ref{ev} 
        that the relative Euler class of $\phi\in \HP(\Sigma)$ 
        can be determined by evaluating
        the commutator and the product maps on $\univcover\times \univcover$. 
        As a generalization to \textcolor{black}{a surface $\Sigma$ with}
        arbitrary Euler characteristic $\chi(\Sigma)\leqslant -1$, we introduce the \emph{evaluation map} $\ev: \hom\to \univcover$, which is defined as follows. 
        Let $\Sigma = \Sigma_{g,p}$ with $p\geqslant 1$ and $\chi(\Sigma)\leqslant -1$, whose fundamental group has the presentation 
        $$\fund = \langle a_1,b_1,\cdots, a_g, b_g, c_1,\cdots, c_p \ |\ [a_1,b_1]\cdots[a_g, b_g]c_1\cdots c_p \rangle.$$
        Let $\phi\in \hom$.
        For $j\in \{1,\cdots, g\}$, let $\widetilde{\phi(a_j)}$ and $\widetilde{\phi(b_j)}$ be arbitrary lifts of $\phi(a_j)$ and $\phi(b_j)$, respectively, in $\univcover$; and for $i\in \{1,\cdots, p-1\}$, let $\widetilde{\phi(c_i)}$ be the unique lift of $\phi(c_i)$ in $\overline{\Hyp_0}\cup \Ell_1$.
        Then the evaluation $\ev(\phi)$ of $\phi$ is defined by
        $$\ev(\phi)\doteq [\widetilde{\phi(a_1)},\widetilde{\phi(b_1)}]\cdots[\widetilde{\phi(a_g)}, \widetilde{\phi(b_g)}]\widetilde{\phi(c_1)}\cdots \widetilde{\phi(c_{p-1})}.$$
        
        The following Proposition \ref{prop_evimage} describes the image of the evaluation map in $\univcover$.
        \begin{proposition}\label{prop_evimage}
        Let $\Sigma = \Sigma_{g,p}$ with $p\geqslant 1$,
        and let $\phi\in \hom$ be a representation that maps $c_1,\cdots, c_{p-1}$ to positive parabolic elements. Assume further that $\phi$ sends $c_p$ and every non-peripheral simple closed curve in $\pi_1(\Sigma)$ to either a hyperbolic or an elliptic element.
        Then the evaluation $\ev(\phi)$
        lies in \textcolor{black}{$\Hyp_n \cup \Ell_n$} for some $n\in \mathbb{Z}$ \textcolor{black}{satisfying} $1-2g\leqslant n\leqslant 2g + p - 2$. 
    \end{proposition} 
    \begin{proof}
        Since $\ev(\phi)$ projects to $\phi(c_p)^{-1}$ in $\psl$, $\ev(\phi)\in \Hyp_n$ for some $n$ if $\phi(c_p)$ is hyperbolic, and $\ev(\phi)\in \Ell_n$ for some $n$ if $\phi(c_p)$ is elliptic. 
        In the case where $\phi(c_p)$ is hyperbolic, the relative Euler class $e(\phi)$ is well-defined, and by Theorem \ref{thm_exist_condition},
        its value is bounded by $1-2g\leqslant e(\phi)\leqslant 2g + p - 2$. By the definition of $\ev(\phi)$ and \textcolor{black}{of} the relative Euler class, $e(\phi) = n$ if and only if $\ev(\phi)$ lies in $\Hyp_n$, which completes the proof.
        In the case where $\phi(c_p)$ is elliptic, 
        we will proceed by induction on $p$. First, we will consider the base case $p = 1$. Next, for \textcolor{black}{$p\geqslant 2$ and} $g\geqslant 1$, assuming the proposition holds for $\Sigma_{g, p-1}$, we will prove the proposition for $\Sigma = \Sigma_{g, p}$. Finally, we will address the case \textcolor{black}{where $p\geqslant 2$ and} $g = 0$.
        \smallskip 
        
        If $p = 1$,
        then $g\geqslant 1$ as $\chi(\Sigma) = 1 - 2g\leqslant -1$.
        \textcolor{black}{W}e will proceed by induction on $g$. For the base case $g = 1$, 
        the proposition follows directly from Theorem \ref{thm_liftcommu}. 
        Assume that the proposition holds for $\Sigma_{g-1,1}$. 
        Let 
        $\gamma_1 = [a_1,b_1]\cdots[a_{g-1},b_{g-1}]\in \pi_1(\Sigma)$ and 
        $\gamma_2 = [a_{g},b_{g}]\in \pi_1(\Sigma)$. Then the curves $\gamma_1$ and $\gamma_2$ separate $\Sigma$
        into three subsurfaces $P$, $\Sigma_1$, and $\Sigma_2$, where 
        $P\cong \Sigma_{0,3}$,
        $\Sigma_1\cong \Sigma_{g-1,1}$, and $\Sigma_2\cong \Sigma_{1,1}$\textcolor{black}{; see Figure \ref{fig:Prop_3.7}(a).} Applying the induction hypothesis to $\Sigma_1$, 
        we have $$\ev\big(\phi|_{\pi_1(\Sigma_1)}\big) = [\widetilde{\phi(a_1)},\widetilde{\phi(b_1)}]\cdots[\widetilde{\phi(a_{g-1})}, \widetilde{\phi(b_{g-1})}]\in \textcolor{black}{\Hyp_{n_1}\cup \Ell_{n_1}}$$ 
        for some $n_1\in \{3-2g,\cdots, 2g - 3\}$. Also, applying Theorem \ref{thm_liftcommu} to $\Sigma_2$, we have $$\ev\big(\phi|_{\pi_1(\Sigma_2)}\big) = [\widetilde{\phi(a_g)},\widetilde{\phi(b_g)}]
        \in \textcolor{black}{\Hyp_{n_2}\cup \Ell_{n_2}}$$
        for some $n_2\in \{-1, 0, 1\}$.
        Let $m:\univcover\times \univcover\to \univcover$ be the product map sending $(\widetilde{A}, \widetilde{B})$ to $\widetilde{A}\widetilde{B}$.
        Since $\ev(\phi)$ equals the product $$\ev(\phi) = m\Big(\ev\big(\phi|_{\pi_1(\Sigma_1)}\big),\  \ev\big(\phi|_{\pi_1(\Sigma_2)}\big)\Big)
        = \ev\big(\phi|_{\pi_1(\Sigma_1)}\big)\cdot  \ev\big(\phi|_{\pi_1(\Sigma_2)}\big),$$  it lies in one of 
    \begin{eqnarray*}
        m(\Hyp_{n_1}\times \Hyp_{n_2})
        &=& z^{n_1+n_2}\cdot m(\Hyp_0\times \Hyp_0)
        ,\\
        m(\Hyp_{n_1}\times \Ell_{n_2})
         &=& \begin{cases}
            z^{n_1+n_2-1}\cdot m(\Hyp_0\times \Ell_1) & \mbox{ if }n_2 = 1 \\
            z^{n_1+n_2}\cdot m(\Hyp_0\times \Ell_1) & \mbox{ if }n_2 = -1
        \end{cases},\\
        m(\Ell_{n_1}\times \Hyp_{n_2})
        &=& 
        \begin{cases}
            z^{n_1+n_2-1}\cdot m(\Ell_1\times \Hyp_0) & \mbox{ if }n_1\geqslant 1 \\
            z^{n_1+n_2}\cdot m(\Ell_1\times \Hyp_0) & \mbox{ if }n_1\leqslant -1
        \end{cases},  
        \text{ and} \\
        m(\Ell_{n_1}\times \Ell_{n_2}) &=&
        \begin{cases}
            z^{n_1+n_2-1}\cdot m(\Ell_{-1}\times \Ell_1) & \mbox{ if }n_1\geqslant 1\mbox{ and }n_2 = 1, \\
            z^{n_1+n_2+1}\cdot m(\Ell_{-1}\times \Ell_1) & \mbox{ if }n_1\leqslant -1\mbox{ and }n_2 = -1,\\
            z^{n_1+n_2}\cdot m(\Ell_{-1}\times \Ell_1) & \text{\textcolor{black}{if} otherwise.}
        \end{cases}
    \end{eqnarray*}
        Since $\ev(\phi)$ is elliptic, by Theorem \ref{thm_prodimage} and 
        Lemma \ref{lem_evimage_base}, $\ev(\phi)\in  \Ell_{n_1+n_2-1}\cup\Ell_{n_1+n_2}\cup \Ell_{n_1+n_2+1}$ where $2-2g\leqslant n_1+n_2\leqslant 2g-2$. Therefore, the proposition holds for all $\Sigma_{g,1}$ with $g\geqslant 1$.
        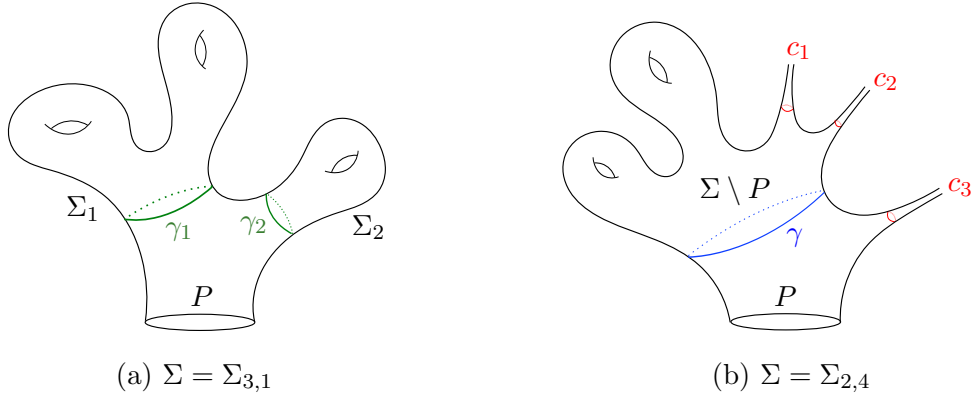
\begin{figure}[hbt!]
        \centering
        \begin{overpic}[width=1.0
        \textwidth]{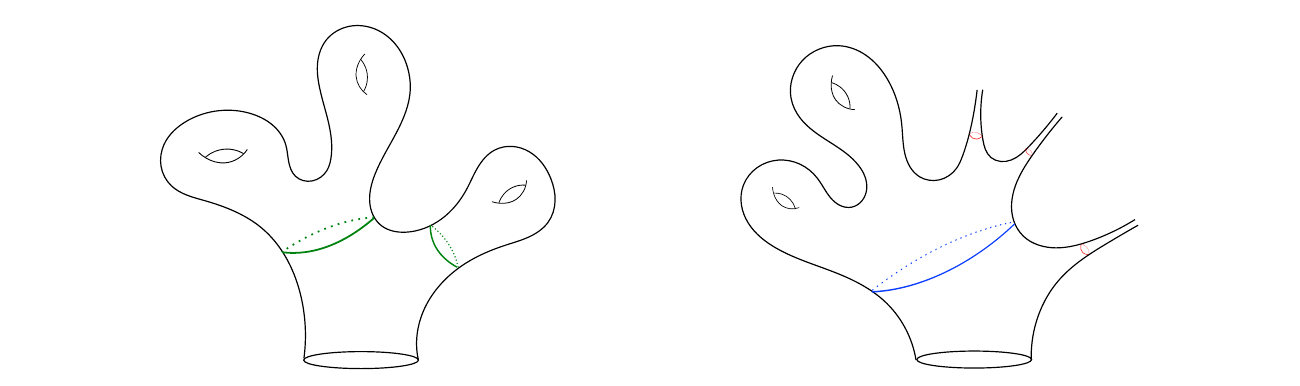}
            \put(25,9){\textcolor{OliveGreen}{$\gamma_1$}}
            \put(31,9.5){\textcolor{OliveGreen}{$\gamma_2$}}

            \put(75,8.5){\textcolor{blue}{$\gamma$}}

            \put(75,23.5){\textcolor{red}{$c_1$}}

            \put(82,21){\textcolor{red}{$c_2$}}

            \put(88,12.5){\textcolor{red}{$c_3$}}

            \put(27,3.3){\textcolor{black}{$P$}}

            \put(17,10.5){\textcolor{black}{$\Sigma_1$}}

            \put(40,9){\textcolor{black}{$\Sigma_2$}}

            \put(74,3.3){\textcolor{black}{$P$}}

            \put(68,12){\textcolor{black}{$\Sigma\setminus P$}}
            \put(21,-3){\textcolor{black}{(a)  $\Sigma = \Sigma_{3,1}$}}

            \put(69,-3){\textcolor{black}{(b)  $\Sigma = \Sigma_{2,4}$}}

 \end{overpic}
\vspace{0.3em}        \caption{\label{fig:Prop_3.7} Decomposition of $\Sigma$ for \textcolor{black}{the} induction.}
    \end{figure}

       We now let \textcolor{black}{$p\geqslant 2$ and $g\geqslant 1$}, and assume that the proposition holds for $\Sigma_{g, p-1}$.
        Similar to the case $p = 1$, we let $\gamma = [a_1,b_1]\cdots[a_g,b_g]c_1\cdots c_{p-2}$, separating $\Sigma$ into two subsurfaces $P$ and $\Sigma\setminus P$, 
        where 
        $P\cong \Sigma_{0,3}$ and
        $\Sigma\setminus P\cong \Sigma_{g,p-1}$\textcolor{black}{; see Figure \ref{fig:Prop_3.7}(b).} Applying the induction hypothesis to $\Sigma\setminus P$, we have 
        $$\ev\big(\phi|_{\pi_1(\Sigma\setminus P)}\big) = [\widetilde{\phi(a_1)},\widetilde{\phi(b_1)}]\cdots[\widetilde{\phi(a_g)}, \widetilde{\phi(b_g)}]\widetilde{\phi(c_1)}\cdots \widetilde{\phi(c_{p-2})}\in \Ell_{n'}\cup \Hyp_{n'}$$ 
        for some $n'\in \{1-2g,\cdots,  2g + p - 3\}$. 
        Then the value $\ev(\phi) = m\big(\ev(\phi|_{\pi_1(\Sigma\setminus P)}),\  \widetilde{\phi(c_{p-1})}\big)$ lies in either 
        $$m(\Hyp_{n'}\times \Par_0^+)
        = z^{n'}\cdot m(\Hyp_0\times \Par_0^+)
        $$
        or
        $$m(\Ell_{n'}\times \Par_0^+) =
        \begin{cases}
            z^{n'-1}\cdot m(\Ell_1\times \Par_0^+) & \mbox{ if }n'\geqslant 1 \\
            z^{n'}\cdot m(\Ell_1\times \Par_0^+) & \mbox{ if }n'\leqslant -1.
        \end{cases}
        $$
        Since $\ev(\phi)$ is elliptic, Theorem \ref{thm_prodimage} and 
        Lemma \ref{lem_evimage_base} 
        imply that $\ev(\phi)\in \Ell_{n'}\cup \Ell_{n'+1}$. 
        Therefore, the proposition holds for all $\Sigma_{g,p}$ with $g\geqslant 1$ and $p\geqslant 1$.

        Finally, we address the case $g = 0$. 
        \textcolor{black}{In this case,} $p\geqslant 3$ as $\chi(\Sigma) = 2 - p \leqslant -1$.
        we will proceed by induction on $p$. For the base case $p = 3$, 
        the proposition follows directly from Theorem \ref{thm_prodimage}. 
        For $p\geqslant 4$, assuming that the proposition holds for $\Sigma_{0,p-1}$, the induction follows verbatim the previous case $g\geqslant 1$ and $p\geqslant 2$. The proof of Proposition \ref{prop_evimage} is now complete.
        \end{proof}

    To prove Theorem \ref{thm_tothyp}, we also need the following Lemma \textcolor{black}{from \cite{goldman_torus}}.
    \begin{lemma}\label{lem_torus}\cite[Lemma 3.4.5]{goldman_torus} 
        Let $A, B\in \SL$. The following conditions are equivalent:
        \begin{enumerate}[(1)]
       \item $\tr[A,B]<2$;
       \item $A, B$ are hyperbolic elements and their invariant axes cross.
       \end{enumerate}
    \end{lemma}
    We now prove Theorem \ref{thm_tothyp}.
    \begin{proof}[Proof of Theorem \ref{thm_tothyp}]
        Since $\phi\in \NPns$, for any simple closed curve $\gamma\in \pi_1(\Sigma)$, $\phi(\gamma)$ is  either hyperbolic or elliptic. Assume, for 
        \textcolor{black}{the sake of}
        contradiction, that 
        $\phi(\gamma)$ is elliptic. 
        We will derive a contradiction separately in the cases where $\gamma$ is non-separating and  where $\gamma$ is separating, thereby concluding that $\phi(\gamma)$ is hyperbolic.
        \smallskip

        If $\gamma$ is non-separating, 
        then there is a non-separating simple closed curve $\delta$ whose geometric intersection number with $\gamma$ equals $1$. 
        Let $T$ be a one-holed torus containing $\gamma$ and $\delta$, 
        whose boundary component \textcolor{black}{represents} the commutator $[\gamma, \delta]$. 
        Let $\widetilde{\phi(\gamma)}$ and $\widetilde{\phi(\delta)}$ be arbitrary lifts of $\phi(\gamma)$ and $\phi(\delta)$, respectively, in $\univcover$. Consider
        their commutator $[\widetilde{\phi(\gamma)}, \widetilde{\phi(\delta)}]$,
        which lies in 
        $\mathcal{I}\subset \univcover$ defined in Theorem \ref{thm_liftcommu}.
        Since $\phi(\gamma)$ is elliptic, by Lemma \ref{lem_torus}, the projection of $[\widetilde{\phi(\gamma)}, \widetilde{\phi(\delta)}]$ to $\SL$ has trace at least 2, and therefore $[\widetilde{\phi(\gamma)}, \widetilde{\phi(\delta)}]$ lies in $\overline{\Hyp_0}\subset \mathcal{I}$. By the definition of the relative Euler class, we have $e(\phi|_{\pi_1(T)}) = 0$; and by Proposition \ref{prop_additivity}, $e(\phi|_{\pi_1(\Sigma\setminus T)}) = n = -\chi(\Sigma\setminus T)$. However, as all peripheral elements of $\pi_1(\Sigma)$ are contained in $\pi_1(\Sigma\setminus T)$ and $p_-(s) = 1$, $\phi|_{\pi_1(\Sigma\setminus T)}$ sends some peripheral element of $\pi_1(\Sigma\setminus T)$ to a negative parabolic element. Hence by Theorem \ref{thm_exist_condition}, $e(\phi|_{\pi_1(\Sigma\setminus T)}) \leqslant -\chi(\Sigma\setminus T) - 1$, \textcolor{black}{which is} a contradiction. 
        Therefore, $\phi(\gamma)$ must be hyperbolic.
        \smallskip
        
        If $\gamma$ is separating, 
        then it separates $\Sigma$ into two subsurfaces 
        $\Sigma_1$ and $\Sigma_2$, where $\Sigma_1\cong \Sigma_{j, k+1}$ and $\Sigma_2 \cong \Sigma_{g-j, p-k+1}$ for some $j\in \{0,\cdots, g\}$ and $k\in \{0,\cdots, p\}$. Choose primitive peripheral elements $c_1, \cdots, c_p$ of $\pi_1(\Sigma)$ so that the fundamental groups of $\Sigma_1$ and $\Sigma_2$ are presented as 
        $$\pi_1(\Sigma_1) = \langle a_1, b_1, \cdots, a_j, b_j,c_1,\cdots, c_k,\gamma\ |\ [a_1, b_1] \cdots, [a_j, b_j]c_1\cdots c_k \gamma  \rangle$$
        and 
        $$\pi_1(\Sigma_2) = \langle a_{j+1}, b_{j+1}, \cdots, a_g, b_g,c_{k+1},\cdots, c_{p},\gamma\ |\ [a_{j+1}, b_{j+1}] \cdots, [a_g, b_g]c_{k+1}\cdots c_p \gamma^{-1}  \rangle.$$
        Up to a permutation of the peripheral elements, and after possibly reindexing $\Sigma_1$ and $\Sigma_2$, we can assume that 
        $s_i = +1$ for all $i\in \{1,\cdots, p-1\}$ and $s_p = -1$. 
        Consider the evaluations 
        $\ev\big(\phi|_{\pi_1(\Sigma_1)}\big)$ and $\ev\big(\phi|_{\pi_1(\Sigma_2)}\big)$, which are respectively the lifts of $\phi(\gamma)^{-1}$ and $\phi(\gamma)$ in $\univcover$.
        Since $\phi(\gamma)$ is elliptic, 
        for $i\in \{1,2\}$,         $\ev\big(\phi|_{\pi_1(\Sigma_i)}\big)$ lies in $\Ell_{n_i}$ for some $n_i\in \mathbb{Z}$. 
        We \textcolor{black}{will} first determine the possible values of $n_1$ and $n_2$ using Proposition \ref{prop_evimage}. 
        \textcolor{black}{Then} by comparing $\ev\big(\phi|_{\pi_1(\Sigma_1)}\big)$ and $\ev\big(\phi|_{\pi_1(\Sigma_2)}\big)$, we \textcolor{black}{will} conclude that the relative Euler class of $\phi$ cannot be $n = -\chi(\Sigma) - 1$, which leads to a contradiction.

        As $\phi(c_1), \cdots,\phi(c_k)$ are positive parabolic,
        by Proposition \ref{prop_evimage}, the evaluation
        $$\ev\big(\phi|_{\pi_1(\Sigma_1)}\big) = [\widetilde{\phi(a_1)},\widetilde{\phi(b_1)}]\cdots[\widetilde{\phi(a_j)}, \widetilde{\phi(b_j)}]\widetilde{\phi(c_1)}\cdots \widetilde{\phi(c_k)}$$ lies in $\Ell_{n_1}$ for an $n_1\in \{1-2j,\cdots, 2j + k - 1\}$.
       To determine the possible values of $n_2$, we first consider the curve $\gamma c_p^{-1}$ that separates $\Sigma_2$ into two subsurfaces $P$ and $\Sigma_2\setminus P\cong \Sigma_{g-j, p-k}$. 
        Similar to $\ev\big(\phi|_{\pi_1(\Sigma_1)}\big)$ above, 
        as 
        $\phi(c_{k+1}), \cdots,$ and $\phi(c_{p-1})$ are positive parabolic and $\phi\big(\gamma c_p^{-1})$ is either elliptic or hyperbolic,
        the evaluation
        $$\ev(\phi|_{\pi_1(\Sigma_2\setminus P)}) = [\widetilde{\phi(a_{j+1})},\widetilde{\phi(b_{j+1})}]\cdots[\widetilde{\phi(a_g)}, \widetilde{\phi(b_g)}]\widetilde{\phi(c_{k+1}})\cdots \widetilde{\phi(c_{p-1})}$$
        lies in $\Ell_{n'_2}\cup \Hyp_{n'_2}$ for an 
        $n'_2\in \{1-2(g-j), \cdots, 2(g-j) + p-k- 2\}$.
        Let $\widetilde{\phi(c_p)}$ be the lift of $\phi(c_p)$ in $\Par^-_0$, and observe that 
        $\ev\big(\phi|_{\pi_1(\Sigma_2)}\big) = \ev(\phi|_{\pi_1(\Sigma_2\setminus P)})\cdot\widetilde{\phi(c_p)}$. If $\ev\big(\phi|_{\pi_1(\Sigma_1\setminus P)}\big)\in \Ell_{n'_2} = z^{n'_2-1}\Ell_1$, then by Theorem \ref{thm_prodimage},
        $n_2 = n'_2$; and \textcolor{black}{if} $\ev\big(\phi|_{\pi_1(\Sigma_1\setminus P)}\big)\in \Hyp_{n'_2} = z^{n'_2}\Hyp_0$, \textcolor{black}{then} by Lemma \ref{lem_evimage_base}, 
        $n_2 = 
        n'_2 \mbox{ if } n'_2\geqslant 1$, and 
        $n_2 = n'_2-1 \mbox{ if } n'_2\leqslant -1$. Thus,
        $n_2\in \{-2(g-j), \cdots, 2(g-j) + p-k- 2\}$.
        \smallskip
        
        By the definition of the relative Euler class $n  = e(\phi)$, 
        we have $\ev\big(\phi|_{\pi_1(\Sigma_2)}\big)= z^n\ev\big(\phi|_{\pi_1(\Sigma_1)}\big)^{-1}$, which implies that 
        $\Ell_{n_2} = z^n\Ell_{-n_1}$. 
        Since $n_1\leqslant 2j + k - 1$ and $n_2 \leqslant 2(g-j) + p-k- 2$, and since $\Ell_{n_1} = z^{n_1+n_2 - 1}\Ell_{-n_2}$ when $n_1, n_2$ are both positive, 
        we conclude that 
        $$n \textcolor{black}{ \ = n_1 + n_2 - 1}
        \leqslant  (2j + k - 1) + \big(2(g-j) + p-k- 2\big) - 1 \\
        \textcolor{black}{ \ = 2g + p - 4} = -\chi(\Sigma) - 2,$$
        contradicting that $n = -\chi(\Sigma) - 1$. Therefore, $\phi(\gamma)$ is hyperbolic. This completes the proof.
    \end{proof}

    \subsection{Proof of Corollary \ref{cor_almost_fuchsian}}\label{proof_cor}
    Finally, we prove Corollary \ref{cor_almost_fuchsian}, which we restate below as Corollary \ref{cor2_almost_fuchsian}.
    \begin{corollary}\label{cor2_almost_fuchsian}
        Let $\Sigma = \Sigma_{g,p}$ with $\chi(\Sigma)\leqslant -2$ and $p\geqslant 1$,  and let $n = -\chi(\Sigma) - 1$ and $s\in \{\pm 1\}^p$ satisfying $p_-(s) = 1$. 
        For every $\phi\in \NPns$, 
        there exists a pair of pants $P$ in $\Sigma$ such that,
        for each connected component $\Sigma'$ of the complement $\Sigma\setminus P$,
        the restriction $\phi|_{\pi_1(\Sigma')}$ is Fuchsian.
    \end{corollary}
    \begin{proof}
        Let $c_1,\cdots, c_p$ be 
        the primitive peripheral elements of $\pi_1(\Sigma)$. As $p_-(s) = 1$,
        there is a unique $i\in \{1,\cdots, p\}$ such that $\phi(c_i)$ is negative parabolic.
        Let $P$ be a pair of pants in $\Sigma$ whose fundamental group $\pi_1(P)$ contains $c_i$.
        By Theorem \ref{thm_tothyp}, every common boundary component of $P$ and $\Sigma\setminus P$ maps to a hyperbolic element of $\psl$. Hence $\phi|_{\pi_1(P)}\in \HP(P)$, and for each connected component $\Sigma'$ of the complement $\Sigma\setminus P$, $\phi|_{\pi_1(\Sigma')}\in \HP(\Sigma')$.
        Moreover, Theorem \ref{thm_exist_condition} implies that $e(\phi|_{\pi_1(P)})\leqslant0$ and $e(\phi|_{\pi_1(\Sigma\setminus P)}) \leqslant -\chi(\Sigma\setminus P) = n$, and by Proposition \ref{prop_additivity}, $e(\phi|_{\pi_1(P)})=0$ and $e(\phi|_{\pi_1(\Sigma\setminus P)}) = n$. Then for each connected component $\Sigma'$ of $\Sigma\setminus P$, we have $e\big(\phi|_{\pi_1(\Sigma')}\big) = -\chi(\Sigma')$. Consequently, by Proposition \ref{prop_holonomy}, $\phi|_{\pi_1(\Sigma')}$ is a holonomy representation, hence Fuchsian.
    \end{proof}

\noindent
Inyoung Ryu\\
Department of Mathematics\\  Texas A\&M University\\
College Station, TX 77843, USA\\
(riy520@tamu.edu)


\begin{thebibliography}{99}

\setlength{\itemsep}{0pt}
\setlength{\parsep}{0pt}


    




\bibitem{bowditch}
    B.H. Bowditch,
    \textit{Markoff triples and quasifuchsian groups},
    Proc. London Math. Soc., Vol 77 (1998), 697–736.

    
    \bibitem{goldman_measure}
    W. Goldman,
    \textit{The symplectic nature of fundamental groups of surfaces},
    Adv. in Math. 54 (1984),
    200–225.


    
\bibitem{goldman}
    W. Goldman,
    \textit{Topological components of spaces of representations},
    Invent. Math. 93 (1988), 
    no.3, 
    557–607.


\bibitem{goldman_torus}
    W. Goldman,
    \textit{The modular group action on real SL(2)-characters of a one-holed torus},
    Geom. Topol., 7 (2003), 443–486. 

\bibitem{goldman_conjecture}
    W. Goldman,
    \textit{Mapping class group dynamics on surface group representations},
    Proc. Symp. Pure Math. vol. 74 (2006), 189–214.

\bibitem{Mityagin}
    B. Mityagin,
    \textit{The zero set of a real analytic function},
    Math Notes, 107 (2020), 529–530.


\bibitem{MW}
    J. Marché and M. Wolff, \textit{The modular action on $\psl$-characters in genus 2},
    Duke Math., J. 165(2) (2016), 371-412.

\bibitem{MW2}
    J. Marché and M. Wolff, \textit{Six-point configurations in the hyperbolic plane and ergodicity of the mapping class group}, Groups Geom. Dyn. 13 (2019), no. 2, pp. 731–766.


\bibitem{RY}
    I. Ryu and T. Yang,
    \textit{Connected components of the space of type-preserving representations}, 2025, preprint.
    











\bibitem{yang}
    T. Yang,
    \textit{On type-preserving representations
of the four-punctured sphere group},
    Geom. Topol. 20 (2016),
    1213–1255.

\end{thebibliography}
\end{document}